\documentclass[12pt]{article}
\usepackage{booktabs}
\usepackage{caption}
\usepackage{mathrsfs}
\usepackage{amsmath}
\usepackage{amsfonts,amsthm,amssymb,mathrsfs,bbding}
\usepackage{graphics,multicol}
\usepackage{graphicx}
\usepackage{color}
\usepackage{enumerate}
\usepackage{caption}
\usepackage{rotating}
\usepackage{lscape}
\usepackage{longtable}
\allowdisplaybreaks[4]
\usepackage{tabularx}
\usepackage[colorlinks=true,anchorcolor=blue,filecolor=blue,linkcolor=blue,urlcolor=blue,citecolor=blue]{hyperref}
\usepackage{extarrows}
\usepackage{cite}
\usepackage{latexsym,bm}
\usepackage{mathtools}
\evensidemargin=\oddsidemargin\topmargin=-1.5cm

\newtheorem{thm}{Theorem}

\newtheorem{lem}[thm]{Lemma}
\newtheorem{cor}[thm]{Corollary}

\theoremstyle{definition}

\addtocounter{section}{0}
\begin{document}

	\title{\bf Spectral radius conditions for  fractional $[a,b]$-covered graphs}
	\author{{Junjie Wang$^{a}$, Jiaxin Zheng$^{a,}$\footnote{Corresponding author.}\setcounter{footnote}{-1}\footnote{\emph{E-mail address:} w2451329132@163.com (J. Wang), fabulousxjz@163.com (J. Zheng), C2202840320@163.com (Y. Chen).}, \  Yonglei Chen$^{b}$ }\\[2mm]
		\small $^{a}$School of Mathematics, East China University of Science and Technology\\
		\small Shanghai 200237, China \\
		\small $^{b}$Institute of Applied Mathematics, Lanzhou Jiaotong University\\
		\small Lanzhou 730070, China}

	\date{}
	\maketitle
	{\flushleft\large\bf Abstract }  A graph $G$ is called fractional $[a,b]$-covered if for every edge $e$ of $G$ there is a fractional $[a,b]$-factor with the indicator function $h$ such that $h(e)=1$. In this paper, we provide tight spectral radius conditions for graphs being fractional $[a,b]$-covered. 
	
	\begin{flushleft}
	\end{flushleft}

	\section{Introduction}
	All graphs considered in this paper are simple and undirected. Let $G$ be a graph with vertex set $V(G)$ and  edge set $E(G)$. Let $e(G):=|E(G)|$ denote the number of edges in $G$.  For any $v\in V(G)$, let $d_G(v)$ denote the degree of $v$ in $G$, $N_G(v)$  denote the set of vertices adjacent to $v$ in $G$, and $E_G(v)$ denote the set of edges incident with $v$ in $G$.  For any vertex subset $S\subseteq V(G)$, we denote by $G[S]$ the subgraph of $G$ induced by $S$, and $e(S):=e(G[S])$. Also, we denote by $e(S,T)$ the number of edges between two disjoint subsets $S$ and $T$ of $V(G)$.  A vertex set $S\subseteq V(G)$ is called \emph{independent} if any two vertices in $S$ are non-adjacent in $G$. The \textit{join} of two graphs $G_1$ and $G_2$, denoted by $G_1\nabla G_2$,  is the graph obtained from the vertex-disjoint union $G_1\cup G_2$ by adding all possible edges between $G_1$ and $G_2$.

The \textit{adjacency matrix} of $G$ is defined as $A(G)=(a_{u,v})_{u,v\in V(G)}$, where $a_{u,v}=1$ if $u$ and $v$ are adjacent in $G$, and $a_{u,v}=0$ otherwise. Let $D(G)=\mathrm{diag}\{d_G(v): v\in V(G)\}$ denote the diagonal degree matrix of $G$. Then the \textit{signless Laplacian matrix} of $G$ is defined as $Q(G)=D(G)+A(G)$.  The largest eigenvalues of  $A(G)$ and $Q(G)$ are called the \textit{spectral radius} and  \textit{signless Laplacian spectral radius} of $G$, and denoted by $\rho(G)$ and $q(G)$, respectively. For some basic  bounds of $\rho(G)$ and $q(G)$, we refer the reader to \cite{CS,CS2,CS3,Z,H,HZ05,NP,STE,WXH}, and references therein.

Let $f$ and $g$ be two integer-valued functions defined on $V(G)$ such that $0 \le f(x) \le g(x)$ for all $x \in V (G)$, and let  $h : E(G) \rightarrow [0, 1]$ be a function defined on $E(G)$ satisfying $f(x) \le \sum_{e \in E_G(x)}h(e) \le g(x)$  for all $x \in V (G)$. Setting $F_{h} = \{e: e \in E(G),h(e) > 0\}$. Then the subgraph of $G$ with vertex set $V(G)$ and edge set $F_h$, denoted by  $G[F_h]$, is called a \textit{fractional $(f, g)$-factor} of $G$ with indicator function $h$. If for each edge $e$ of $G$, there is a fractional $(f, g)$-factor with the indicator function $h$, such that $h(e)=1$, then $G$ is called a  \textit{fractional $(f,g)$-covered graph}. In particular, if $f(x)=a$ and $g(x)=b$  for all $x \in V(G)$ ($a,b$ are positive integers with $a\leq b$),  then a fractional $(f, g)$-factor is called  a \textit{fractional $[a,b]$-factor}, and a fractional $(f,g)$-covered graph is called a \textit{fractional $[a,b]$-covered graph}. For more notions  about factors of graphs, see \cite{GLI,MS,HH,ST,Z3,Z4,SFL,L1,L2,DAE,ND,J}. 

In \cite{LYZ}, Li, Yan and Zhang introduced the concept of fractional $(f, g)$-covered graphs, and gave a necessary and sufficient condition for a graph being fractional $(f, g)$-covered. As an immediate corollary, we obtain the following result.

\begin{thm}(Li, Yan and Zhang \cite{LYZ})\label{thm::1}
		Let $b \ge a \ge 1$ be two integers. Then a graph $G$ is fractional  $[a, b]$-covered if and only if
		\begin{equation}\label{equ::1}
		    \delta_{G}(S,T)=b|S|-a|T|+\sum_{x\in T}d_{G-S}(x)\ge \varepsilon(S)
		\end{equation}
		for every vertex subset $S$ of $G$, where $T=\{x:x\in V(G)\backslash S, d_{G-S}(x)\le a\}$ and $\varepsilon(S)$ is defined by
		\begin{equation}\label{equ::2}
		\varepsilon(S)=\left\{
		\begin{array}{ll}
			2,&\mbox{if $S$ is not independent},\\
			1,&\mbox{if $S$ is independent, and there exists $e=uv\in E(G)$ with}\\
			&\mbox{$u \in S$, $v \in T$ and $d_{G-S}(v)=a$, or $e_{G}(S, V(G)\backslash (S\cup T))\ge 1$,}\\
			0,&\mbox{otherwise}.
		\end{array}
		\right.
		\end{equation}
	\end{thm}

Based on Theorem \ref{thm::1}, Yuan and Hao \cite{YH1} witnessed a degree condition for fractional $[a,b]$-covered graphs. 	
	\begin{thm}(Yuan and Hao\cite{YH1})\label{thm::4}
	Let $3\le a\le b$ be integers, and let $G$ be a graph of order $n$ with minimum degree not less than a+1. Suppose that $n\geq  ((a+b)(a+b-2)+a)/b$ when $a \ge 4$ and $n\geq  ((a+b)(a+b-3/2)+a)/b$ when $a=3$. 
 If $G$ satisfies $\max \{d_{G}(x), d_{G}(y) \} \ge a(n+1)/(a+b)$ for each pair of nonadjacent vertices $x$, $y$ of $G$, then $G$ is a fractional  $[a,b]$-covered graph.
\end{thm}
	
Very recently, Yuan and Hao \cite{YH2} presented a neighborhood union condition for a graph being fractional $[a,b]$-covered. 
	\begin{thm}(Yuan and Hao\cite{YH2})\label{thm::5}
	Let $2\le a\le b$ and $r\ge 2$ be integers, and let $G$ be a graph of order $n$ with $n>((a+b)(r(a+b)-2)+a)/b$ and $\delta(G)\ge (r-1)(a+1)^2/a$. If $G$ satisfies $\mid N_G(x_1)\cup N_G(x_2)\cup \cdots \cup N_G(x_r)\mid \ge a(n+1)/(a+b)$ for any independent subset $\{x_1,x_2,\ldots,x_r\}$ of $G$, then $G$ is a fractional  $[a,b]$-covered graph.
\end{thm}

In this paper, by using Theorem \ref{thm::1}, we provide tight spectral radius conditions for a graph being fractional $[a,b]$-covered. For any integers $a$ and $n$ with $2\leq a\leq n$, we denote $H_{n,a}:=K_{a-1}\nabla (K_1\cup K_{n-a})$. The main results of this paper are as follows.

	\begin{thm}\label{thm::2}
		Let $b \ge a \ge 2$ be two integers, and $G$ be a graph of order $n \geq 2+\sqrt{32a^2+24a+5}$. If $\rho (G) \geq \rho (H_{n,a})$, then $G$ is a fractional $[a,b]$-covered graph unless $G\cong H_{n,a}$.
	\end{thm}

	\begin{thm}\label{thm::3}
		Let $b \ge a \ge 2$ be two integers, and $G$ be a graph of order $n \geq 6a+5$. If $q (G) \geq q (H_{n,a})$, then $G$ is a fractional $[a,b]$-covered graph unless $G\cong H_{n,a}$.
	\end{thm}

	\section{Preliminaries}
	
	In this section, we introduce some notions and lemmas, which are useful in the proof of the main results.

        Let $M$ be a real $n$ $\times$ $n$ matrix, and let $\Pi=\{X_{1},X_{2}, \ldots,X_{k}\}$ be a partition of $[n]=\{1,2,\ldots,n\}$. Then the matrix $M$ can be written as
        $$
	M=\left(\begin{array}{ccccccc}
		M_{1,1}&M_{1,2}&\cdots&M_{1,k}\\
            M_{2,1}&M_{2,2}&\cdots&M_{2,k}\\
            \vdots&\vdots&\ddots&\vdots\\
            M_{k,1}&M_{k,2}&\cdots&M_{k,k}\\
	\end{array}\right).
	$$
 The \textit{quotient matrix} of $M$ with respect to $\Pi$ is  the matrix $B_{\Pi}=(b_{i,j})^{k}_{i,j=1}$ with
        $$
            b_{i,j}=\frac{1}{|X_{i}|}\mathbf j^{T}_{|X_{i}|}M_{i,j}\mathbf j_{|X_{j}|}
        $$
for all $i,j \in \{ 1,2,...,k \}$, where $\mathbf j_{s}$ denotes the all ones vector in $\mathbb{R}^{s}$. If each block $M_{i,j}$ of $M$ has constant row sum $b_{i,j}$, then $\Pi$ is called an \textit{equitable partition}, and the quotient matrix $B_\Pi$ is called an  \textit{equitable quotient matrix} of $M$. Also, if the eigenvalues of $M$ are all real, we denote them by $\lambda_1(M)\geq \lambda_2(M)\geq \cdots \geq \lambda_n(M)$. 

        \begin{lem}(Brouwer and Haemers \cite[p. 30]{AW}; Godsil and Royle \cite[pp.196--198]{GR})\label{lem::1}
            Let $M$ be a real symmetric matrix, and let $B$ be an equitable quotient matrix of $M$. Then the eigenvalues of $B$ are also eigenvalues of $M$. Furthermore, if $M$ is nonnegative and irreducible, then
            $$
                \lambda_1(M)=\lambda_1(B).
            $$
	\end{lem}

	\begin{lem}(Hong \cite{H})\label{lem::2}
		Let $G$ be a connected graph with $n$ vertices and $m$ edges. Then
		$$\rho(G)\le \sqrt{2m-n+1}.$$
	\end{lem}

	\begin{lem}(Feng and Yu \cite{Z})\label{lem::3}
		Let $G$ be a connected graph with $n$ vertices and $m$ edges. Then
		$$q(G)\le \frac{2m}{n-1}+n-2.$$
	\end{lem}
	
    \begin{lem}\label{lem::4}
    The graph $H_{n,a}$ with $n\geq a+3$ is not a fractional $[a, b]$-covered graph.
    \end{lem}
    \begin{proof}
   Recall that $H_{n,a}=K_{a-1}\nabla (K_1\cup K_{n-a})$. Let $V_1=V(K_1)$, $V_2=V(K_{a-1})$ and $V_3=V(K_{n-a})$.  Suppose  $S=\emptyset$ and $T=V_1$. Clearly, $\varepsilon(S)=0$ by \eqref{equ::2}. Also note that  $T$ contains all vertices of degree at most $a$ in $H_{n,a}-S=H_{n,a}$ because $n\geq a+3$. Furthermore,  we have
   $$\delta_G(S,T)=b|S|-a|T|+\sum_{x\in T}d_{G-S}(x)=-a+a-1=-1<\varepsilon(S),$$
which violates the inequality in \eqref{equ::1}. Therefore, by Theorem \ref{thm::1}, we conclude that  $H_{n,a}$ is not a fractional $[a, b]$-covered graph.
    \end{proof}

    \begin{lem}\label{lem::5}
Let $n$ and $a$ be positive integers with $n\geq \sqrt{32a^2+24a+5}+2$. Then
$$\rho(K_{4a+1} \nabla (K_{2} \cup K_{n-4a-3}))\leq n-2.$$
    \end{lem}
    \begin{proof}
Suppose $L_{n,a}=K_{4a+1} \nabla (K_{2} \cup K_{n-4a-3})$. Let $V_{1}=V(K_2)$, $V_2=V(K_{4a+1})$ and $V_3=V(K_{n-4a-3})$. Then it is easy to see that the partition  $\Pi: V(L_{n,a})=V_1\cup V_2\cup V_3$ is an equitable partition of $L_{n,a}$, and the corresponding quotient matrix is 	$$
	B_\Pi=\left(\begin{array}{ccccccc}
		1&4a+1&0\\
		2&4a&n-4a-3\\
		0&4a+1&n-4a-4\\
	\end{array}\right).
	$$
 Let $f(x)$ denote the characteristic polynomial of $B_\Pi$. By a simple computation, we have
	$$
		\begin{aligned}
			f(n-2)=|(n-2)I-B| = n^{2}-4n-32a^{2}-24a-1 \geq 0
			\nonumber
		\end{aligned}
	$$
	because $n\geq \sqrt{32a^2+24a+5}+2$. We claim that $\lambda_1(B_\Pi)\leq n-2$. If not, since $f(n-3)=-2(4a+1)^2<0$, we have $\lambda_3(B_\Pi)>n-3$, and hence $\lambda_1(B_\Pi)+\lambda_2(B_\Pi)+\lambda_3(B_\Pi)> 3n-9$. On the other hand,  $\lambda_1(B_\Pi)+\lambda_2(B_\Pi)+\lambda_3(B_\Pi)=\mathrm{trace}(B_\Pi)=n-3$, we obtain a contradiction. Therefore, by Lemma \ref{lem::1},
	$$
	\rho(L_{n,a})=\lambda_1(B_\Pi)\leq n-2,
	$$
	and our results follows.
    \end{proof}
    By using a similar method, one can easily deduce the following result.
        \begin{lem}\label{lem::6}
Let $n$ and $a$ be positive integers with $n\geq 6a+5$. Then
$$q(K_{4a+1} \nabla (K_{2} \cup K_{n-4a-3}))\leq 2n-4.$$
    \end{lem}

	\section{Proof of the main results}
	
In this section, we shall prove Theorems  \ref{thm::2} and \ref{thm::3}.

	{\flushleft \it Proof of Theorem \ref{thm::2}.}
	By assumption, we have $\rho(G) \geq \rho(H_{n,a})>\rho(K_{n-1})=n-2$ because $K_{n-1}$ is a proper subgraph of $H_{n,a}$. We claim that  $G$ is connected. If not, then each component of $G$ would be a subgraph of $K_{n-1}$, and hence $\rho (G) \le \rho(K_{n-1}) = n-2$, a contradiction.
	
	Suppose to the contrary that $G$ is not a fractional $[a, b]$-covered graph and $G\ncong H_{n,a}$. By Theorem \ref{thm::1}, there exists some subset $S\subseteq V(G)$ such that
	\begin{equation}\label{equ::3}
	    \delta_{G}(S, T) = b|S|-a|T|+\sum_{x \in T}d_{G-S}(x) \le \varepsilon(S)-1,
	\end{equation}
	where $T=\{x:x\in V(G)\backslash S, d_{G-S}(x)\le a\}$ and $\varepsilon(S)$ is defined in \eqref{equ::2}. Let $s=|S|$ and $t=|T|$. As $b\geq a\geq 2$, from \eqref{equ::2} and \eqref{equ::3} one can easily deduce that $t>0$ and $s \le t$. We consider the following two cases.
	
	{\flushleft {\it Case 1.} $t=1$.}

	In this situation, suppose $T=\{x_{0}\}$. As $s\le t$, we have $s=0$ or $1$. If $s=0$, i.e., $S=\emptyset$, then  $\varepsilon(S)=0$ according to \eqref{equ::2}, and it follows from \eqref{equ::3} that $d_G(x_{0})=d_{G-S}(x_{0}) \le a-1$. Thus $G$ is a spanning subgraph of $H_{n,a}$.  If $s=1$, then \eqref{equ::3} gives that $d_{G-S}(x_{0}) \le \varepsilon(S)-1+a-b$. Note that $\varepsilon(S) \le 1$ by \eqref{equ::2} and the fact that $|S|=s=1$. Thus $d_{G-S}(x_{0}) \le a-b \le a-2$, and  $G$ is also a spanning subgraph of $H_{n,a}$.  As  $G\ncong H_{n,a}$, in both cases, we obtain $\rho(G)<\rho(H_{n,a})$, contrary to our assumption.

	{\flushleft {\it Case 2.} $t\geq 2$.}
	
	First we claim that $t\leq 2a+2$. By contradiction,  suppose that $t\geq 2a+3$. According to \eqref{equ::2} and \eqref{equ::3}, we have
	$\sum_{x \in T}d_{G-S}(x) \le 1+at-bs$. Let $T'=V(G)\setminus(S\cup T)$. Then
	\begin{equation}
		\begin{aligned}
		e(G)&=e(S)+e(S,T)+e(S,T')+e(T)+e(T,T')+e(T')\\
		 &\le \frac{s(s-1)}{2}+st+s(n-s-t)+\sum_{x \in T}d_{G-S}(x)+\frac{(n-s-t)(n-s-t-1)}{2}\\
		 &\le \frac{s(s-1)}{2}+st+s(n-s-t)+(1+at-bs)+\frac{(n-s-t)(n-s-t-1)}{2}\\
		&=\frac{(n-2)^{2}-n(2t-3)+t^{2}+t+2at-2bs+2st-2}{2}.
		\nonumber
		\end{aligned}
	\end{equation}
	Since $n\geq s+t$ and $t\geq 2a+3$, by Lemma \ref{lem::2}, we obtain
	\begin{equation}
		\begin{aligned}
			\rho(G)&\le \sqrt{2e(G)-n+1}\\
			&\le  \sqrt{(n-2)^{2}-2n(t-1)+t^{2}+t+2at-2bs+2st-1}\\
			&\le  \sqrt{(n-2)^{2}-2(s+t)(t-1)+t^{2}+t+2at-2bs+2st-1}\\
			&=\sqrt{(n-2)^{2}-(t^{2}-(2a+3)t+2(b-1)s+1)}\\
			&\le n-2\\
			&< \rho(H_{n,a}),
			\nonumber
		\end{aligned}
	\end{equation}
	contrary to our assumption. Hence, $t \le 2a+2$. Let $G_{1}= K_{4a+1} \nabla (K_{2} \cup K_{n-4a-3})$. We shall prove that $G$ is a spanning subgraph of $G_1$. In fact, by definition, every vertex in $T$ has degree at most $a$ in $G-S$. Furthermore, we assert that there exists some vertex $x_1\in T$ such that $d_{G-S}\leq a-1$, since otherwise we can deduce from \eqref{equ::3} that $\delta_G(S,T)=bs \le \varepsilon(S)-1$, which is impossible by \eqref{equ::2} and the fact that $b\geq 2$. As $|T|=t\geq 2$, we can choose $x_2\in T$ with $x_2\neq x_1$. Recall that $|S|=s\leq t\leq 2a+2$. Then we have
	$|(N_G(x_1)\setminus\{x_2\})\cup (N_G(x_2)\setminus\{x_1\})|\leq |S|+|(N_{G-S}(x_1)\setminus\{x_2\})\cup (N_{G-S}(x_2)\setminus\{x_1\})|\leq (2a+2)+(a-1)+a=4a+1$, and hence $G$ is a spanning subgraph of $G_1$. Combining this with Lemma \ref{lem::5}, we obtain
	$\rho(G)\leq \rho(G_1)\leq n-2<\rho(H_{n,a})$, contrary to our assumption.
	
Note that $H_{n,a}$ is not a fractional $[a,b]$-covered graph by Lemma \ref{lem::4}. Therefore, we conclude that $G$ is a fractional $[a,b]$-covered graph unless $G\cong H_{n,a}$. \qed

	{\flushleft \it Proof of Theorem \ref{thm::3}.}
	As in Theorem \ref{thm::2}, we have $q(G) \geq q(H_{n,a})>q(K_{n-1})=2n-4$ and  $G$ is connected. By contradiction, suppose that $G$ is not a fractional $[a, b]$-covered graph and $G\ncong H_{n,a}$. Then there exists some subset $S\subseteq V(G)$ satisfying \eqref{equ::3}, where $T=\{x:x\in V(G)\backslash S, d_{G-S}(x)\le a\}$ and $\varepsilon(S)$ is defined in \eqref{equ::2}. Let $s=|S|$ and $t=|T|$. We have $t>0$ and $s \le t$. If $t=1$, by the analysis in Theorem \ref{thm::2}, we deduce that $\rho(G)<\rho(H_{n,a})$, a contradiction. If $t\geq 2a+3$, as in  Theorem \ref{thm::2}, from Lemma \ref{lem::3} we obtain
	\begin{equation}
	\begin{aligned}
			q(G)&\le \frac{2e(G)}{n-1}+n-2\\
&\leq \frac{(n-2)^{2}-n(2t-3)+t^{2}+t+2at-2bs+2st-2}{n-1}+n-2\\
			&=  2n-4-\frac{2bs-t-2at+2n(t-1)-2st-t^{2}}{n-1}\\
&\le  2n-4-\frac{2bs-t-2at+2(s+t)(t-1)-2st-t^{2}}{n-1}\\
			&=  2n-4-\frac{t^{2}-(2a+3)t+2bs-2s}{n-1}\\
			&\le  2n-4\\
			&< q(H_{n,a}),
			\nonumber
		\end{aligned}
	\end{equation}
which is impossible. Hence, $2\leq t \le 2a+2$. Again by the analysis in Theorem \ref{thm::2}, we assert that $G$ is a spanning subgraph of $G_1=K_{4a+1} \nabla (K_{2} \cup K_{n-4a-3})$. Then, by Lemma \ref{lem::6},
	$q(G)\leq q(G_1)\leq 2n-4<q(H_{n,a})$, a contradiction. Therefore, we conclude that $G$ is a fractional $[a,b]$-covered graph unless $G\cong H_{n,a}$. \qed

	\section{Concluding remarks}

In this paper, we provide tight spectral radius conditions for a graph being fractional $[a,b]$-covered. In \cite{GL}, Liu and Zhang gave a necessary and sufficient condition for the existence of a fractional $[a,b]$-factor in a graph.
	\begin{thm}(Liu and Zhang \cite{GL})\label{thm::7}
		Let $b \ge a \ge 1$ be two integers, and $G$ be a graph. Then $G$ has a fractional $[a,b]$-factor if and only if for every subset S of V(G)
		$$b|S|-a|T|+\sum_{x\in T}d_{G-S}(x)\ge 0,$$
		where $T=\{x:x\in V(G)\backslash S, d_{G-S}(x)\le a\}$.
    \end{thm}
According to Theorem \ref{thm::7}, as in Lemma \ref{lem::4}, it is easy to see that the graph $H_{n,a}$ with $n\geq 2a+3$ has no $[a,b]$-factor. Since each fractional $[a,b]$-covered graph must contain a fractional $[a,b]$-factor, by Theorems \ref{thm::2} and \ref{thm::3}, we obtain the following two results, respectively.

	\begin{cor}\label{cor::1}
		Let $b \ge a \ge 2$ be two integers, and $G$ be a graph of order $n \geq 2+\sqrt{32a^2+24a+5}$. If $\rho (G) \geq \rho (H_{n,a})$, then $G$ has a fractional $[a,b]$-factor unless $G\cong H_{n,a}$.
    \end{cor}

    \begin{cor}\label{cor::2}
    	Let $b \ge a \ge 2$ be two integers, and $G$ be a graph of order $n \geq 6a+5$. If $q (G) \geq q (H_{n,a})$, then $G$ has a fractional $[a,b]$-factor unless $G\cong H_{n,a}$
    \end{cor}

%

\end{document}